\documentclass{azjmPRE_2}

\usepackage{enumerate}
\usepackage[pdftex,unicode,colorlinks,linkcolor=blue,
citecolor=red,bookmarksopen,pdfhighlight=/N]{hyperref}

\newcommand{\const}{{\rm const}}

\renewcommand{\leq}{\leqslant}
\renewcommand{\geq}{\geqslant}

\def\const{\mathrm{const}}
\def\RR{\mathbb R}

\def\NN{\mathbb N}

\DeclareMathOperator{\clos}{clos}
\DeclareMathOperator{\Int}{int}
\DeclareMathOperator{\Har}{har}
\DeclareMathOperator{\Hol}{Hol}

\DeclareMathOperator{\Zero}{Zero}
\DeclareMathOperator{\sbh}{sbh}

\DeclareMathOperator{\supp}{supp}
\DeclareMathOperator{\comp}{c}

\begin{document}

\title{Uniqueness Theorems for Subharmonic and Holomorphic Functions of Several Variables on a Domain
}

\author[B.\,N. Khabibullin, N.\,R. Tamindarova]{ B.\,N. Khabibullin, N.\,R. Tamindarova\footnote{Our research is supported by RFBR grant (project no.~16-01-00024).}}

\begin{abstract}
We establish a general uniqueness theorem for  subharmonic  functions of several varia\-b\-l\-es on a domain. A corollary from this  
uniqueness theorem for holomorphic functions is formulated in terms of the zero subset of holomorphic functions and re\-s\-t\-r\-ictions on the growth of functions near the boundary of domain.
\end{abstract}

\keywords{holomorphic function, zero set, uniqueness set, subharmon\-ic functi\-on, Rie\-sz measure, Jensen measure, potential, balayage\\
}

\ams{32A10, 	31B05}

\maketitle

\section{Introduction}
\subsection{Definitions and notations}

We use an information and definitions from \cite{HK}--\cite{H}.
As usual, $\mathbb N:=\{1,2, \dots\}$, $\mathbb R$ and $\mathbb C$ are the sets of all natural, real and complex numbers, resp. 
We set  
\begin{equation}\label{df:R}
\RR^+:= \{x\in \RR\colon x\geq 0\},\; \RR_{-\infty}:=\{-\infty\}\cup \RR,\; 	\RR_{+\infty}:=\RR\cup \{+\infty\}, 
\; \RR_{\pm\infty}:=\RR_{-\infty}\cup \RR_{+\infty},
\end{equation} 
where the usual  order relation $\leq$ on $\RR$  is complemented by the inequalities
$-\infty \leq x\leq +\infty$ for all $x\in \RR_{\pm\infty}$. 
Let $f\colon X\to Y$ be a function.  Given  $S\subset X$, we denote by $f\bigm|_S$ the restriction of $f$ to $S$. 
For $Y\subset \RR_{\pm\infty}$, $g\colon X\to \RR_{\pm\infty}$ and $S\subset X$, we write 
``$f \leq g$ {\it on\/} $S$\,'' if $f(x)\leq g(x)$ for all $x\in S$.

Let $m\in \NN$. Denote by $\mathbb R^m$ the {\it  $m$-dimensional Euclidian real  space.\/} Then $\mathbb R^m_{\infty}
:=\mathbb R^m \cup \{\infty\}$ is the {\it  Alexandroff\/} ($\Leftrightarrow$one-point) {\it compactification of\/} $\mathbb R^m$.
Given a subset $S$ of $\mathbb R^m$ (or $\mathbb R^m_{\infty}$), the closure $\clos S$, the interior $\Int S$  and the boundary $\partial S$ will always be taken relative $\mathbb R^m_{\infty}$. Let $S_0\subset S\subset \mathbb R^m_{\infty}$. If  the  closure $\clos S_0$ 
 is a compact subset of $S$ in the topology induced on $S$  from $\mathbb R^m_{\infty}$, then the set $S_0$ 
is a {\it relatively compact subset} of $S$, and we write $S_0\Subset S$.

Let $n\in \NN$. Denote by $\mathbb C^n $ the {\it  $n$-dimensional Euclidian complex  space.\/} Then $\mathbb C^n _{\infty}:=\mathbb C^n \cup \{\infty\}$ is the {\it  Alexandroff\/} ($\Leftrightarrow$one-point) {\it compactification of\/} $\mathbb C^n$.  If it is necessary,  
we identify $\mathbb C^n $ (or $\mathbb C^n _{\infty}$) with $\mathbb R^{2n}$ (or $\mathbb R^{2n}_{\infty}$). Given a subset $S$ of $\mathbb C^n$ (or $\mathbb C^n_{\infty}$), its closure $\clos S$ and its boundary $\partial S$ will always be taken relative to $\mathbb C^n _{\infty}$. 

Let $A, B$ are sets, and $A\subset B$. The set $A$   is a {\it non-trivial subset\/} of the set $B$ if the subset $A\subset B$ is non-empty ($A\neq \varnothing$) and   {\it proper\/} ($A\neq B$).
	
	We understand always the {\it ``positivity''\/} or {\it ``positive''\/} as $\geq 0$, where  the symbol $0$ denotes the number zero, the zero function, the zero measure, etc. So, a function 	$	f\colon X\to R\overset{\eqref{df:R}}{\subset} \RR_{\pm\infty}$ 
	is positive on $X$ if 	$f(x)\geq 0$ for all $x\in X$. In such case we write ``$f\geq 0$ {\it on\/} $X$''.

The class of all Borel real measures on local compact space $X$ is denoted by $\mathcal M(X)$,  
$\mathcal M_{\comp}(X)$ is the subclass of all Borel measures $\mu$ on $X$ with compact support $\supp \mu \subset X$,
and $\mathcal M^+(X) \subset \mathcal M(X)$  is the subclass of all Borel positive  measures on $X$,  
$\mathcal M_{\comp}^+(X) :=\mathcal M^+(X) \cap \mathcal M_{\comp}(X)$. 
Given  $\mu\in  \mathcal M(X) $ and $S\subset X$, we denote by $\mu\bigm|_S$ the restriction of $\nu$ to $S$. For $\nu \in \mathcal M(X)$, we write ``$\nu \geq \mu$ {\it on\/} $S$\,'' if $\bigl(\nu\bigm|_S-\mu\bigm|_S\bigr)\in \mathcal M^+(S)$.

Let $\mathcal O$ be a non-trivial  open subset of $\mathbb R^m _{\infty}$. We denote by $\sbh (\mathcal O)$ the class of  all subharmonic functions $u\colon \mathcal O\to \RR_{-\infty}$ on $\mathcal O$ for $m\geq 2$, and all (local) convex functions  $u\colon \mathcal O\to \RR_{-\infty}$ on $\mathcal O$ for $m=1$. The class $\sbh (\mathcal O)$  contains the function 
$\boldsymbol{-\infty}\colon x\mapsto -\infty$, $x\in \mathcal O$  (identical to $-\infty$); $\sbh^+(\mathcal O):=
\{u\in \sbh (\mathcal O)\colon u\geq 0 \text{ on $\mathcal O$}\}$. We set $\sbh_*(\mathcal O):=\sbh\,(\mathcal O)\setminus \{\boldsymbol{-\infty}\}$. For $u\in \sbh_*(\mathcal O)$, the {\it  Riesz measure of\/} $u$ is the  Borel  positive 
 measure 
\begin{equation}\label{df:cm}
	\nu_u:= c_m \,\Delta u\in \mathcal M^+(\mathcal O),  \quad c_m:=\frac{\Gamma(m/2)}{2\pi^{m/2}\max\bigl\{1, (m-2)\bigr\}}\,,
\end{equation}
where $\Delta$ is  the {\it Laplace operator\/}  acting in the sense of distribution theory, and $\Gamma$ is the gamma function. 
Such measures $\nu_u$ is Radon measures, i.\,e. $\nu_u(S)<+\infty$ for each subset $S\Subset \mathcal O$.
By definition, $\nu_{\boldsymbol{-\infty}}(S):=+\infty$ for all $S\subset \mathcal O$. 

Let $\mathcal O$ be a non-trivial  open subset of $\mathbb C^n_{\infty}$. We denote by $\Hol (\mathcal O)$ and $\sbh (\mathcal O)$ the class of holomorphic  and subharmonic functions on $\mathcal O$, resp.
For $u\in \sbh_*(\mathcal O)$, the {\it  Riesz measure of\/} $u$ is the  Borel (and  Radon) positive measure 
\begin{equation*}
	\nu_u:=c_{2n} \, \Delta u\in \mathcal M^+(\mathcal O), \quad c_{2n}=\frac{(n-1)!}{2\pi^n \max\{1,2n-2\}}\,.
\end{equation*}

\subsection{Main Theorem and Corollary}

\begin{definition} Let $D$ be a  non-trivial  open connected  subset of\/ $\mathbb R^m_{\infty}$, i.\,e. $D$ is a non-trivial domain in\/  $\mathbb R^m_{\infty}$. Let  $K$   be a non-trivial compact subset of $D$, i.\,e. $\varnothing \neq K=\clos K\subset D$.
A function $v\in \sbh^+ (D\setminus K)$ is called a  test function for $D$ outside of $K$ if 
\begin{equation}\label{v0l}
\lim_{D\ni x'\to x} v(x')=0 \quad \text{for each  $x\in \partial D$}
\quad \text{and}\quad  \sup_{x\in D\setminus K}v(x)<+\infty.
\end{equation}
The class of all such  test functions for $D$ outside of $K$ is denoted by $\sbh_0^+(D\setminus K)$.
\end{definition}

Our main result for subharmonic functions is the following 
\begin{theorem}[{see \cite[Corollary 1.1]{KhT15} for the case $m=2$}]\label{th:s} 
Let\/ $D$ be a non-trivial domain in $\RR^m_{\infty}$, $K$ a  compact subset of $D$
with non-empty interior $\Int K\neq \varnothing$. Let  $M\in \sbh_*(D)$ be a function  with the Riesz measure $\nu_M\in \mathcal M^+(D)$, $ v\in \sbh_0^+(D\setminus K)$ a test function for $D$ outside of $K$. Assume that
\begin{equation}\label{co:M}
	\int_{D\setminus K} v\, {\rm d} \nu_M<+\infty.
\end{equation}
If $u\in \sbh(D)$ is a function  with the Riesz measure $\nu_u \in \mathcal M^+(D)$ such that
\begin{subequations}\label{s:}
\begin{align} 
\nu_u&\geq \nu\in \mathcal M^+(D\setminus K) \text{ on } D\setminus K, 
\tag{\ref{s:}a}\label{s:a}\\
\quad  \int_{D\setminus K}v\, {\rm d} \nu&=+\infty,
\tag{\ref{s:}b}\label{s:b}
\\
u&\leq M+\const  \text{ on } D,
\tag{\ref{s:}c}\label{s:c}
\end{align}
\end{subequations}
where\/ $\const$ is a constant, then $u=\boldsymbol{-\infty}$.
\end{theorem}

We denote by $ \sigma_{2n-2}$ the {\it $(2n-2)$-dimensional surface\/} ($\Leftrightarrow$Hausdorff) {\it measure\/} on $\mathbb C^n$
and its restrictions to subsets of $\mathbb C^n$. So, if $n=1$, i.\,e. $2n-2=0$, then $\sigma_0(S)=\sum_{z\in S}1$ for each $S\subset \mathbb C$, i.\,e. $\sigma_0(S)$ is equal to the number of points in the set $S\subset \mathbb C$.
 
Below we identify $\mathbb C^n $ (or $\mathbb C^n _{\infty}$) with $\mathbb R^{m}$ (or $\mathbb R^{m}_{\infty}$) where 
 $m=2n$.

Our main result for holomorphic  functions is the following 

\begin{corollary}\label{th:h} Let all conditions of Theorem\/ {\rm \ref{th:s}} are fulfilled including\/ \eqref{co:M}. 
Let  $f\in \Hol (D)$ be  a holomorphic function on $D$ and 
$\Zero_f:=\{z\in D \colon f(z)=0\}$. 
If 
\begin{subequations}\label{h:}
\begin{align} 
{\tt Z}&\subset(D\setminus K) \cap  \Zero_f,
\tag{\ref{h:}a}\label{h:a}\\
\int_{\tt Z}v\, {\rm d} \sigma_{2n-2}&=+\infty,
\tag{\ref{h:}b}\label{h:b}
\\
|f|&\leq \const \,e^M\text{ on } D,
\tag{\ref{h:}c}\label{h:c}
\end{align}
\end{subequations}
where $\const$ is a constant, then $f\equiv 0$ on $D$, i.\,e. $\Zero_f=D$.
\end{corollary}

\begin{proof} Under the conditions of Corollary \ref{th:h}, suppose that $f\neq 0$. Then we have $\log |f|\in \sbh_* (D)$ with the Riesz measure $\nu_{\log |f|}\in \mathcal  M^+(D)$. Let $n_f\colon D\to \{0\}\cup \NN$ be the {\it multiplicity function of\/} $f$ \cite[4]{Chee}. It is known that $\supp n_f=\Zero_f$. By the classical Poincar\'e\,--\,Lelong formula \cite{L} we have
	$n_f\,{\rm d}\sigma_{2n-2}={\rm d} \nu_{\log |f|}$ on $D$. 
	Hence, if the condition \eqref{h:a} is fulfilled then 	we get
\begin{equation}\label{int:Z}
\int_{\tt Z}v\, {\rm d} \sigma_{2n-2}\overset{\eqref{h:a}}{\leq} \int v \,n_f\, {\rm d} \sigma_{2n-2}
\leq \int v\, {\rm d}\nu_{\log |f|}. 
\end{equation}
If the condition \eqref{h:c}   is also fulfilled, then for $u:=\log |f|$ we have \eqref{s:c}  
together with \eqref{s:a} for $\nu=\nu_u=\nu_{\log |f|}$. 
Since $u\neq \boldsymbol{-\infty}$, by Theorem \ref{th:s} we obtain  the negation  of the equality \eqref{s:b}.
Therefore we get
\begin{equation*}
\int_{\tt Z}v\, {\rm d} \sigma_{2n-2}\overset{\eqref{int:Z}}{\leq} \int v\,{\rm d} \nu_{\log|f|}=\int v\,{\rm d} \nu_{u}<+\infty
\end{equation*}
 what contradicts \eqref{h:b}. Corollary \ref{th:h} is proved.~{$\blacktriangleleft$}
\end{proof}

\section{Main results}

\subsection{Gluing Theorem for $m\in \NN$}

The next result shows how two subharmonic functions can be glued together. 
\begin{theorem}[{\rm see \cite[Corollary 2.4.5]{Klimek}, and \cite[Theorem 2.4.5]{R} for $m=2$}]\label{th:g}
Let $\mathcal  O, O_0$ are open sets in $\RR^m_{\infty}$, and 
$\mathcal O \subset \mathcal O_0$. Let $v_0\in \sbh (\mathcal O_0)$, and $v\in \sbh (\mathcal O)$.  If 
\begin{equation}\label{0vs}
	\limsup_{\mathcal O\ni  x'\to x} v(x')\leq v_0(x) \quad \text{for all points $x\in \mathcal O_0\cap \partial \mathcal O$},
\end{equation}
then the function
\begin{equation}\label{consv}
	\widetilde{v}:=\begin{cases}
\max\{v,v_0\} \quad &\text{on\/ $\mathcal O$},\\
v_0		\quad &\text{on\/ $\mathcal O_0\setminus \mathcal O$},
	\end{cases}
\end{equation}
belong  to the class $\sbh(\mathcal O_0)$. 
\end{theorem}
\begin{remark}\label{r:1} 
{\rm A similar gluing  theorem true also for classes of  plurisubharmonic functions \cite[Corollary 2.9.5]{Klimek}.}
\end{remark}

\subsection{Jensen measures and potentials}
Let $m\in \NN$. Given $t\in \RR_*:=\RR\setminus \{0\}$, we set 
\begin{equation*}
	h_m(t):=
\begin{cases}
|t|\quad &\text{for $m=1$},\\
\log|t|\quad &\text{for $m=2$},\\
-\dfrac{1}{|t|^{m-2}}\quad &\text{for $m\geq 3$}.
\end{cases}
\end{equation*}

For simplicity, we consider only domains $D$ in $\RR^m\subset \RR_{\infty}^m$, i.\,e. $\infty \notin D$.

\begin{definition}[{\rm \cite{Gam}--\cite{Khab03}}]\label{df:mJ} Let $D\subset \RR^{m}$ be a  subdomain, $x_0\in D$. 
A measure $\mu \in \mathcal M_{\comp}^+(D)$ is called the Jensen measure for $\sbh(D)$ 
at  $x_0\in D$  if
\begin{equation*}
	u(x_0)\leq \int u\,{\rm d} \mu \quad\text{for all $u\in \sbh(D)$.}
\end{equation*}
By $J_{x_0}(D)$ we denote the class of all Jensen measures for $D$ at $x_0$. 
Each Jensen measure $\mu\in J_{x_0}(D)$ is a probability measure, i.\,e. $\mu (D)=1$. 

For $\mu\in J_{x_0}(D)$ we shall say that the function 
\begin{equation}\label{df:Vmu}
	V_{\mu}(x):=\int h_m(x-y) \, {\rm d} \mu (y)-h_m(x-x_0), \quad x\in \RR^m_{\infty}\setminus \{x_0\},
\end{equation}
is a potential of the Jensen measure $\mu\in J_0(D)$. 

A function $V\in \sbh^+\bigl(\RR^m_{\infty}\setminus \{x_0\}\bigr)$ is called the Jensen potential   inside of $D$ with pole at $x_0\in D$ if 
 the following two conditions hold:
\begin{enumerate}[{\rm (i)}]
\item\label{V:i}  there is a compact subset $K_V\subset D$ such that $V\equiv 0$ on $\RR^m_{\infty}\setminus K_V$ {\rm (finiteness),}
\item\label{V:ii} and $\limsup\limits_{x_0\neq x\to x_0}\dfrac{V(x)}{\bigl|h_m(x-x_0)\bigr|} \leq 1$
{\rm (semi-normalization at $x_0$).}
\end{enumerate}
By $PJ_{x_0}(D)$ we denote the class of all Jensen potentials inside of $D$ with pole at $x_0\in D$. 
\end{definition}
We present interrelations between Jensen measures and potentials. The first is

\begin{proposition}[{\rm \cite[Proposition 1.4, Duality Theorem]{Khab03}}]\label{pr:1} The map  
\begin{equation}\label{con:P}
	\mathcal P \colon J_{x_0}(D)\to PJ_{x_0} (D), \quad 
	\mathcal P (\mu)\overset{\eqref{df:Vmu}}{:=} V_{\mu}, \quad \mu \in  J_{z_0}(D),
\end{equation}
is the bijection from $J_{x_0}(D)$ to $PJ_{x_0}(D)$ such that 
$\mathcal P\bigl(t\mu_1+(1-t)\mu_2\bigr)=t\mathcal P (\mu_1)+(1-t)\mathcal P (\mu_2)$ for all  $t\in [0,1]$ and for  all $\mu_1,\mu_2\in J_{x_0}(D)$. Besides, 
\begin{equation}\label{eq:mu}
	{{\mathcal P}}^{-1}(V)\overset{\eqref{df:cm}}
	{=}c_m \,\Delta  V\Bigm|_{D\setminus \{x_0\}}+
	\left(1-\limsup\limits_{x_0\neq x \to x_0}\dfrac{V(x )}{|h_m(x-x_0)|}\right)\cdot
	{\delta}_{x_0}\, , \quad V\in PJ_{x_0}(D),
\end{equation}
where $\delta_{x_0}$ is the Dirac measure at the point $x_0$, i.\,e. $\supp \delta_{x_0}=\{x_0\}$ and
$\delta_{x_0}\bigl(\{x_0\}\bigr)=1$.
\end{proposition}

The second is a generalized Poisson\,--\,Jensen formula.
\begin{proposition}[{\cite[Proposition 1.2]{Khab03}}]\label{pr:2}
Let  $\mu \in J_{x_0}(D)$. For each function $u\in \sbh (D)$ with $u(x_0)\neq -\infty$ and the Riesz measure
$\nu_u\in \mathcal M^+(D)$  we have the equality
\begin{equation}\label{f:PJ}
	u(x_0) +\int_{D\setminus \{x_0\}} V_{\mu} \,d {\nu}_u=\int_{D} u \,d \mu  .
\end{equation}
\end{proposition}
Given $x\in \RR^m$ and $r\overset{\eqref{df:R}}{\in} \RR^+$, we set $B(x,r):=\{x'\in \RR^m \colon |x'-x|<r\}$, where $|\cdot |$ is the Euclidean norm  on  $ \RR^m$; $B_*(x,r):=B(x,r)\setminus \{x\}$; $\overline B(x,r):=\clos B(x,r)$. 

By $\Har (\mathcal O)$ denote the class of all harmonic functions on an open subset $\mathcal O\subset \RR^m$.  

\begin{corollary}\label{cr:J} Let $D$ be a domain in $\RR^m$, $x_0\in D$, $r_0>0$, and $B(x_0,r_0)\Subset D$; $b\in \RR^+$.  If  functions   $u\in \sbh_*(D)$ with the  Riesz measure $\nu_u$ and $M\in \sbh_*(D)$ with the Riesz measure $\nu_M$  satisfy the inequality
\begin{equation}\label{uMe}
	u\leq M+\const \quad\text{on $D$},
\end{equation}
then there is a constant $C\in \RR$ such that
\begin{equation}\label{inDuM}
	\int_{D\setminus \overline B(x_0,r_0)} V\,{\rm d}\nu_u\leq \int_{D\setminus \overline B(x_0,r_0)} V \,{\rm d}\nu_M+C
\end{equation}
 for all functions $V\in PJ_{x_0}(D)$ satisfying the  following three  conditions: 
\begin{subequations}\label{cVh}
\begin{align} 
V\bigm|_{B_*(x_0,r_0)}&\in \Har \bigl(B_*(x_0,r_0)\bigr),
\notag
\\ 
 \limsup\limits_{x_0\neq x\to x_0}\dfrac{V(x)}{\bigl|h_m(x-x_0)\bigr|}&\overset{\eqref{V:ii}}{=} 1 \quad\text{\rm (normalization at $x_0$),}
\notag
\\
\sup_{x\in \partial B(x_0,r_0)}V(x)\leq b.
\tag{\ref{cVh}}\label{cVh:b}
\end{align}
\end{subequations}
\end{corollary}
\begin{proof} The technique of balayage out from the ball $B(x_0,r_0)$ gives two functions  $u_0\in \sbh_* (D) $ with the Riesz measure $\nu_{u_0}$ and  $ M_0\in \sbh_* (D) $ with the Riesz measure $\nu_{M_0}$ such that 
\begin{subequations}\label{se:uM}
\begin{align} 
	u_0\bigm|_{B(x_0,r_0)}\in \Har \bigl(B(x_0,r_0)\bigr)&, 
	\quad u_0=u \text{ on } D\setminus B(x_0,r_0),
\tag{\ref{se:uM}u}\label{se:uMu}
\\ 
	M_0\bigm|_{B(x_0,r_0)}\in \Har \bigl(B(x_0,r_0)\bigr)&, 
	\quad  M_0=M \text{ on } D\setminus B(x_0,r_0),
	\notag
	\\
	u_0\overset{\eqref{uMe}}{\leq} M_0+C_0 \quad \text{ on } D&\quad \text{where $C_0\in \RR^+$ is a constant,}
\tag{\ref{se:uM}b}\label{se:uMb}
\\
\supp \nu_{M_0}\subset D\setminus B(x_0,r_0),& \quad \nu_{M_0}\bigl(\partial B(x_0,r_0)\bigl)=
\nu_M \bigl(\overline B(x_0,r_0)\bigl)
\tag{\ref{se:uM}n}\label{se:uMN}
\\
\nu_{M_0}\bigm|_{D\setminus \overline B(x_0,r_0)}&= \quad \nu_{M}\bigm|_{D\setminus \overline B(x_0,r_0)}\,.
\tag{\ref{se:uM}r}\label{se:uMr}
\end{align}
\end{subequations}
By Proposition \ref{pr:1} the measure $\mu\overset{\eqref{eq:mu}}{:=}\mathcal P^{-1} (V)\overset{\eqref{con:P}}{\in} J_{x_0}(D)$ satisfies the following  conditions:
$\supp \mu \overset{\eqref{cVh}}{\subset} D\setminus B(x_0,r_0)$, $\mu(D)=1$.
The inequality \eqref{se:uMb}  entails the inequality
\begin{equation*}
	\int u_0 \,{\rm d} \mu\leq \int M_0 \,{\rm d} \mu+C_0
\end{equation*}
Hence by Proposition \ref{pr:2} 
\begin{equation*}
	\int_{D\setminus \overline B(x_0,r_0)} V\,{\rm d}\nu_{u} \overset{\eqref{se:uMu}}{\leq} \int_D V\,{\rm d}\nu_{u_0} \overset{\eqref{f:PJ}}{\leq} 
	\int_D V\,{\rm d}\nu_{M_0}+\bigl(C_0-u_0(x_0)+M_0(x_0)\bigr)
\end{equation*}
Put $C_1:=C_0-u_0(x_0)+M_0(x_0)\in \RR$.  We continue this inequality as
\begin{multline*}
	\int_{D\setminus \overline B(x_0,r_0)} V\,{\rm d}\nu_{u} 
	\overset{\eqref{se:uMN}}{\leq} \int_{D\setminus \overline B(x_0,r_0)} V\,{\rm d}\nu_{M_0}
	+\int_{ B(x_0,r_0)} V\,{\rm d}\nu_{M_0}+C_1\\
	\overset{\eqref{se:uMr}}{=} 
	\int_{D\setminus \overline B(x_0,r_0)} V\,{\rm d}\nu_{M} +\int_{\partial  B(x_0,r_0)} V\,{\rm d}\nu_{M_0}+C_1\\
	\overset{\eqref{cVh:b}}{\leq} 
	\int_{D\setminus \overline B(x_0,r_0)} V\,{\rm d}\nu_{M}+b\,\nu_{M_0}\bigl(\overline B(x_0,r_0)\bigr)+C_1
	\\
	\overset{\eqref{se:uMN}}{=} 
	\int_{D\setminus \overline B(x_0,r_0)} V\,{\rm d}\nu_{M}+b\,\nu_{M}\bigl(\overline B(x_0,r_0)\bigr)+C_1.
\end{multline*}
We choose $C:=b\,\nu_{M}\bigl(\overline B(x_0,r_0)\bigr)+C_1$ and obtain the inequality \eqref{inDuM}.~{$\blacktriangleleft$}
\end{proof}

\subsection{Continuation of test functions}

\begin{proposition}\label{pr:3}
Let $ v\in \sbh_0^+ (D\setminus K)$ be a  test function for $D$ outside of $K$, $r_0>0$, and $B(x_0,2r_0)\subset K$. Then there are 
subdomains $D_0\Subset D_1\Subset D$,  a number $r_0>0$ and a function $\widetilde{v}\in \sbh^+ \bigl(D\setminus \{x_0\}\bigr)$ 
such that 
\begin{subequations}\label{se:vV+}
\begin{align} 
\overline B(x_0,2r_0) &\subset K \subset D_0, 
\tag{\ref{se:vV+}a}\label{se:vV+a}
\\
\widetilde{v}\bigm|_{B_*(x_0,2r_0)}&\in \Har \bigl(B_*(x_0,2r_0)\bigr), 
\tag{\ref{se:vV+}b}\label{se:vV+b}\\
\exists \; \lim_{x_0\neq x\to x_0}&\dfrac{\widetilde{v}(x)}{\bigl|h_m(x-x_0)\bigr|} \in (0,+\infty),
\tag{\ref{se:vV+}c}\label{se:vV+c}\\
\widetilde{v}=v &\text{ on } D \setminus D_1.
\tag{\ref{se:vV+}d}\label{se:vV+d}
\end{align}
\end{subequations}
\end{proposition}

\begin{proof} Obviously, there is  a  subdomain $D_0\Subset D$ satisfying  \eqref{se:vV+a}. There is  a subdomain $D_1$ that regular for Dirichlet problem and  
$D_0\Subset D_1\Subset D$. Let $g_{D_1}(\cdot , x_0)$ be the Green's function of $D_1$ with the pole $x_0$ where $g_{D_1}(x,x_0)\equiv 0$ for all $x\in D\setminus D_1$. Then $g_{D_1}(\cdot, x_0)\in \sbh \bigl(D\setminus \{x_0\}\bigr)$. 
Put
\begin{equation*}
	q:=\sup_{x\in \partial D_0} v(x), \quad a:=\inf_{x\in \partial D_0}g_{D_1}(x,x_0) >0, \quad v_0:=\frac{q}{a} \, g_{D_1}(\cdot ,x_0), 
\end{equation*}
and $\mathcal O:=D\setminus \clos D_0$, $\mathcal O_0:=D\setminus \{x_0\}$. Then the condition \eqref{0vs} is fulfilled. By Theorem  
\ref{th:g} the function \eqref{consv} is required
according to the known properties of the Green's function.~{$\blacktriangleleft$}
\end{proof}

\subsection{Proof of Theorem \ref{th:s}}
\begin{proof} Let $B(x_0,2r_0)\subset K$ where $r_0>0$. Suppose that  the properties \eqref{s:a}, \eqref{s:c}  are fulfilled for $u\neq \boldsymbol{-\infty}$. 
We must prove that the integral from \eqref{s:b} with  $\nu:=\nu_u$  is finite. 
Consider the function $\widetilde{v}$ from Proposition \ref{pr:3} where
\begin{equation*}
	0<c\overset{\eqref{se:vV+c}}{:=}\lim_{x_0\neq x\to x_0}\dfrac{\widetilde{v}(x)}{\bigl|h_m(x-x_0)\bigr|} <+\infty.
\end{equation*}
Then the function  	$V:=\frac{1}{c}\, \widetilde{v}$  satisfies the following conditions 
\begin{subequations}\label{se:vV+V}
\begin{align} 
V\bigm|_{B_*(x_0,2r_0)}&\overset{\eqref{se:vV+b}}{\in} \Har \bigl(B_*(x_0,2r_0)\bigr), 
\tag{\ref{se:vV+V}b}\label{se:vV+Vb}\\
\lim_{x_0\neq x\to x_0}\dfrac{V(x)}{\bigl|h_m(x-x_0)\bigr|} &\overset{\eqref{se:vV+c}}{=}1,
\tag{\ref{se:vV+V}c}\label{se:vV+Vc}\\
V\overset{\eqref{se:vV+d}}{=} \frac{1}{c}\,v &\text{ on } D \setminus D_1.
\tag{\ref{se:vV+V}d}\label{se:vV+Vd}
\\
\lim_{D\ni x'\to x} V(x')&\overset{\eqref{v0l}}{=}0 \quad \text{for each  $x\in \partial D$},
\tag{\ref{se:vV+V}e}\label{se:vV+Ve}
\\
\sup_{x\in \partial B(x_0,r_0)} V(x)&=:b\overset{\eqref{v0l}}{<}+\infty.
\tag{\ref{se:vV+V}f}\label{se:vV+Vf}
\end{align}
\end{subequations}
We put $V(x)\equiv 0$ at $x\in \RR^m_{\infty}\setminus D$. Then 
	$V\overset{\eqref{se:vV+Ve}}{\in} \sbh^+\bigl(\RR^m_{\infty}\setminus \{x_0\}\bigr)$. Consider the sequence of functions 
$V_n:=\max \{0, V-1/n\}, \quad n\in \NN$.
	If $n_0\in \NN$ is a sufficiently large number, then every function $V_n$, $n\geq n_0$, is a Jensen function inside of $D$ with the pole at $x_0\in D$ such that 
	\begin{subequations}\label{se:vV+V1}
\begin{align} 
V_n\bigm|_{B_*(x_0,r_0)}\overset{\eqref{se:vV+Vb}}{\in} \Har \bigl(B_*(x_0,r_0)\bigr), 
\quad \lim_{x_0\neq x\to x_0}\dfrac{V_n(x)}{\bigl|h_m(x-x_0)\bigr|} &\overset{\eqref{se:vV+Vc}}{=}1,
\quad \sup_{x\in \partial B(x_0,r_0)} V_n(x)\overset{\eqref{se:vV+Vf}}{\leq} b,
\notag
\\
\frac{1}{c}\,v(x)=V(x)\geq V_n(x)\,{\nearrow} \, \frac{1}{c}\,v(x), \quad 
\text{ for all } x\overset{\eqref{se:vV+Vd}}{\in} & D \setminus D_1 \text{ when } n\to \infty.
\tag{\ref{se:vV+V1}}\label{se:vV+V1d}
\end{align}
\end{subequations}
Hence by Corollary  \ref{cr:J} there is a constant $C\in \RR^+$ such that 
\begin{equation*}
	\int_{D\setminus \overline B(x_0,r_0)} V_n\,{\rm d}\nu_u\overset{\eqref{inDuM}}{\leq} \int_{D\setminus \overline B(x_0,r_0)} V_n \,{\rm d}\nu_M+C \quad\text{for all $n\geq n_0$.}
\end{equation*}
Therefore, for all $n\geq n_0$, 
\begin{multline*}
	\int_{D\setminus \overline B(x_0,r_0)} V_n\,{\rm d}\nu_u\overset{\eqref{inDuM}}{\leq} \int_{D\setminus \overline B(x_0,r_0)} V_n \,{\rm d}\nu_M+C 
	\overset{\eqref{se:vV+V1d}}{\leq}\int_{D\setminus \overline B(x_0,r_0)} V \,{\rm d}\nu_M+C \\
\leq \int_{D\setminus D_1} V \,{\rm d}\nu_M+\int_{D_1\setminus \overline B(x_0,r_0)} V \,{\rm d}\nu_M+C 
\overset{\eqref{se:vV+Vf}}{\leq}  \int_{D\setminus D_1} V \,{\rm d}\nu_M+
b \,\nu_M\bigl(D_1\setminus \overline B(x_0,r_0)\bigr) +C 
\end{multline*}
where we use the maximum principle for $V\in \sbh\bigl(\RR_{\infty}^m\setminus \{x_0\}\bigr)$ in $\RR_{\infty}^m\setminus B(x_0,r_0)$.
Further, we put $C_1:=b \,\nu_M\bigl(D_1\setminus \overline B(x_0,r_0)\bigr) +C \in \RR$ and continue as
\begin{equation*}
\int_{D\setminus \overline B(x_0,r_0)} V_n\,{\rm d}\nu_u\leq \int_{D\setminus D_1} V \,{\rm d}\nu_M+C_1
\overset{\eqref{se:vV+V1d}}{=}\frac{1}{c}\,\int_{D\setminus D_1} v  \,{\rm d}\nu_M+C_1
\leq \frac{1}{c}\,\int_{D\setminus K} v  \,{\rm d}\nu_M+C_1.
\end{equation*}
So, when $n\to \infty$, we obtain in view of \eqref{se:vV+V1d}
\begin{equation*}
\frac{1}{c}\int_{D\setminus D_1} v\,{\rm d}\nu_u \overset{\eqref{se:vV+V1d}}{=}\int_{D\setminus D_1} V\,{\rm d}\nu_u \leq C_2:=\frac{1}{c}\,\int_{D\setminus K} v  \,{\rm d}\nu_M+C_1\overset{\eqref{co:M}}{\in} \RR.
\end{equation*}
Hence 
\begin{equation*}
\int_{D\setminus K} v\,{\rm d}\nu_u \leq  \int_{D_1\setminus K} v\,{\rm d}\nu_u+c\,C_2
\leq \nu_u(D_1\setminus K) \sup_{x\in D\setminus K}v(x)+c\,C_2\overset{\eqref{v0l}}{<}+\infty.
\end{equation*}
This completes ерпу proof of Theorem \ref{th:s}.~{$\blacktriangleleft$}
\end{proof}
\begin{remark}\label{r:r3} 
{\rm Our results show that the construction of test functions in the sense of Definition\/ {\rm \ref{df:mJ}}  is important. For $m=2=2n$ such constructions developed in\/ {\rm  \cite{KhT15}.} We consider the case $m\neq 2$ in another place.}
\end{remark}

\begin{footnotesize}
\begin{flushleft}
Bulat N. Khabibullin\\
\textit{Bashkir State University, Ufa, Russian Federation}\\
\textit{E-mail:} khabib-bulat@mail.ru\\
\end{flushleft}
\end{footnotesize}

\begin{footnotesize}
\begin{flushleft}
Nargiza R. Tamindarova\\
\textit{Bashkir State University, Ufa, Russian Federation}\\
\textit{E-mail:} nargiza89@gmail.com\\
\end{flushleft}
\end{footnotesize}

Received ........
 
Accepted ........

\end{document}